\documentclass[a4paper,11pt,oneside]{article}
\usepackage[paperwidth=195mm,paperheight=270mm,left=20mm,right=17mm,top=20mm,bottom=20mm,includefoot]{geometry}

\usepackage{amsmath,amssymb,mathrsfs,amscd,amsthm}

\usepackage{hyperref}

\usepackage{xcolor}
\usepackage{graphicx}
\usepackage{pinlabel}
\usepackage{hyperref}
\usepackage{mathtools}
\usepackage{enumitem}

\newtheorem{Theorem}{Theorem}[section]
\newtheorem{Definition}[Theorem]{Definition}
\newtheorem{Lemma}[Theorem]{Lemma}
\newtheorem{Proposition}[Theorem]{Proposition}
\newtheorem{Corollary}[Theorem]{Corollary}
\newtheorem{Remark}[Theorem]{Remark}

\newcommand{\UMBspace}{\ensuremath{\mathcal{UMB}}}  
\newcommand{\UMBclass}[1][N]{\ensuremath{\mathcal{UMB}(#1)}} 

\usepackage[numbers,sort&compress]{natbib}
\begin{document}
\title{On Uniformly Perfect Morse Boundaries}
\author{Suzhen Han and Qing Liu}

\maketitle

\begin{abstract}
 We introduce and geometrically characterize the notion of uniformly perfect Morse boundary for proper geodesic metric spaces. As a unifying result, we prove that the Morse boundary of any finitely generated, non-elementary group is uniformly perfect whenever it is nonempty. This theorem applies to a broad class of groups, including all acylindrically hyperbolic groups, Artin groups, and hierarchically hyperbolic groups.
 
 Furthermore, we establish a rigidity theorem for homeomorphisms between such boundaries: for any two spaces with uniformly perfect Morse boundaries, a homeomorphism is induced by a quasi-isometry if and only if it satisfies any one of several natural geometric conditions. These conditions include being bi-H\"older, quasi-conformal, quasi-symmetric, or $2$-stable and quasi-M\"obius.
\end{abstract}

\renewcommand{\thefootnote}{\alph{footnote}}
\setcounter{footnote}{-1} \footnote{Keywords: Uniformly perfect, Morse boundary,  Morse geodesically rich, rigidity.}

\section{Introduction and main results}

The Gromov boundary serves as a fundamental invariant of Gromov hyperbolic spaces, bridging their internal geometry with topology. 
A cornerstone result of Gromov \cite{Gromov87} asserts that quasi-isometries between such spaces induce homeomorphisms on their boundaries. Consequently, for a hyperbolic group, the Gromov boundary is a quasi-isometry invariant, well-defined up to homeomorphism.
 
A natural generalization of the Gromov boundary is the Morse boundary $\partial_*X$ of a proper geodesic metric space $X$. It is defined as the set of equivalence classes of Morse geodesic rays, where two rays are equivalent if their images lie within bounded Hausdorff distance. It is classical that all geodesics in a hyperbolic space are uniformly Morse. 
The Morse boundary was first constructed for CAT(0) spaces by Charney and Sultan \cite{CharneySultan15} and later extended to all proper geodesic metric spaces by Cordes \cite{Cordes17}. 
Since then, it has been extensively studied and shown to inherit many key properties of the Gromov boundary \cite{CharneyCordesMurray19, CharneyCordesSisto23, CordesDurham19, CordesHume17, Merlin21, Liu21, Liu22, MousleyRussell19, Murray19, Zbinden23, Zbinden24}. In this paper, we focus on a specific foundational property of this boundary—its \textit{uniform perfectness}.

The concept of \textbf{uniform perfectness} first emerged in conformal dynamics (credited to Beardon and Pommerenke \cite{BeardonPommerenke78, Pommerenke79}) and geometric function theory (Tukia and Väisälä \hspace{0pt}\cite{Tukia80}), where it was originally termed \textit{homogeneous denseness}.
Over subsequent decades, it has proven to be a remarkably versatile property with deep implications across analysis, geometry, and topology \cite{BJ97, Canary91, HR95, JV96, MacManusN-Palka99, MAR18, McMullen90, McMullen98, Sugawa2001, Sugawa01, Vuorinen84}. 

For general metric spaces, several equivalent definitions of uniform perfectness are available \cite{MAR18, Nowak12, Sugawa01}. Notably, Martínez-Pérez and Rodríguez \cite{MAR18} adapted the notion to bounded metric spaces and applied it to Gromov boundaries of hyperbolic manifolds and graphs. More recently, Liang and Zhou \cite{LiangZhou24} used this framework to study hyperbolic spaces.

A fundamental distinction between the Gromov and Morse boundaries lies in topology: the Morse boundary is typically non‑metrizable in its standard topology. This intrinsic feature poses a significant obstacle to directly transplanting the classical theory of uniform perfectness.
Although alternative metrizable topologies have been constructed for the Morse boundaries of certain groups \cite{Cashen19, CordesDussauleGekhtman22, qing24}, the resulting metrics are geometrically opaque and ill‑suited for concrete applications. 
Our work bridges this gap. We introduce a notion of uniform perfectness that is intrinsically adapted to the Morse boundary (Definition~\ref{UP}) and establish equivalent geometric characterizations for it in the setting of proper geodesic metric spaces. Our first main result is the following:
 
\begin{Theorem}\label{thm:equivalence}
Let $X$ be a proper geodesic metric space whose Morse boundary $\partial_*X$ contains at least $3$ points. Then the following statements are equivalent:
\begin{enumerate}
    \item $\partial_*X$ is uniformly perfect and $X$ is a uniformly Morse based space;
    \item $X$ is Morse geodesically rich;
    \item $X$ is center-exhaustive.
\end{enumerate}
Furthermore, each of the above equivalent conditions implies that $X$ is Morse boundary rigid.
\end{Theorem}

The notion of \textbf{geodesic richness} for geodesic metric spaces was first introduced in \cite{Shchur13}.  
Here we extend it to \textbf{Morse geodesic richness} for arbitrary proper geodesic metric spaces (Definition~\ref{Def:GeoRich}).
\textbf{Boundary rigid space}, introduced in \cite{LiangZhou24}, addresses the following natural question: for a hyperbolic space $X$, the canonical homomorphism $\partial\colon\mathcal{QI}(X)\to\operatorname{Homeo}(\partial X)$ from the quasi-isometry group to the homeomorphism group of the boundary—when is it injective? We say $X$ is boundary rigid if $\partial$ is injective. The analogous notion for the Morse boundary is given in Definition~\ref{MBR}. For proper cocompact spaces, we obtain a direct characterization:

\begin{Theorem}\label{thm:cocompact}
Let $X$ be a proper, cocompact geodesic metric space whose Morse boundary $\partial_*X$ contains at least $3$ points. Then $\partial_*X$ is uniformly perfect, and $X$ is both Morse geodesically rich and Morse boundary rigid.
\end{Theorem}

A key link between algebra and geometry is provided by the following fact (Corollary~5.8 in \cite{Liu21}): for a finitely generated group $G$, if $\partial_*G$ is non‑empty and $G$ is not virtually cyclic, then $\partial_G$ contains infinitely many (hence at least three) points.
Combining this with Theorem~\ref{thm:cocompact}, we obtain a unified criterion that ties together algebraic, geometric, and boundary‑theoretic properties:
\begin{Theorem}\label{groupapp}
      Let $G$ be a finitely generated group with non-empty Morse boundary that is not virtually cyclic. Then the Morse boundary $\partial_* G$ is uniformly perfect, and $G$ is Morse geodesically rich and Morse boundary rigid. 
\end{Theorem}

The hypotheses of this theorem are satisfied by a large and prominent class of groups, including (but not limited to): acylindrically hyperbolic groups (including hyperbolic and relatively hyperbolic groups, mapping class groups, and $\text{Out}(F_n)$); Artin groups; Coxeter groups; small cancellation groups; hierarchically hyperbolic groups (including CAT(0) groups, 3-manifold groups, and others); and Morse local-to-global groups.
 See, e.g.,
\cite{CRSZ22, CordesHume17, CharneyCordesSisto23, cordes2025connectivity, GraeberKarrerLazarovich21, han2024growth, MousleyRussell19, Osin16, RussellSprianoTran22, Sisto16, Susse2023, Tran2019, Zbinden23, zbinden23Small, Zbinden24}.

A central theme in geometric group theory is the extent to which a space is determined by its boundary. For Gromov hyperbolic spaces, this question has a rich history \cite{Paulin96, Bourdon95, BonkSchramm00}, and has been extended to more general settings \cite{CharneyCordesMurray19, MousleyRussell19, Liu22}.

The uniform perfectness of the Morse boundary provides a natural framework for extending boundary homeomorphisms to quasi‑isometries. A key novelty of our work is revealing that uniform perfectness is precisely the condition that enables such extensions. Our rigidity theorem is as follows:

\begin{Theorem}\label{homeo}
Let $X, Y$ be proper geodesic metric spaces and let $f: \partial_*X\to \partial_*Y$ be a homeomorphism.
Suppose that each of $X$ and $Y$ satisfies one of the equivalent conditions in Theorem~\ref{thm:equivalence}. Then the following statements are equivalent:
\begin{enumerate}
\item $f$ is induced by a quasi-isometry $F:X\to Y$;
\item $f$ and $f^{-1}$ are bi‑Hölder;
\item $f$ and $f^{-1}$ are quasi‑conformal;
\item $f$ and $f^{-1}$ are quasi‑symmetric;
\item $f$ and $f^{-1}$ are $2$‑stable and quasi‑Möbius.
\end{enumerate}
\end{Theorem}

 The paper is organized as follows. Section~2 reviews the necessary background on Gromov hyperbolic spaces, Morse geodesics, and the Morse boundary. Section~3 is devoted to introducing the key notions of Morse geodesic richness, center-exhaustiveness, and uniform perfectness of the Morse boundary. Section~4 presents the proof of our main characterization theorem (Theorem~\ref{thm:equivalence}). Section~5 demonstrates that these properties are invariant under quasi‑isometry. The final section contains the proof of the rigidity theorem (Theorem~\ref{homeo}).

\vspace{1em}

   \section*{Acknowledgements}
   Both authors thank Wenyuan Yang, Hao Liang, Chenxi Wu, and Qingshan Zhou for helpful conversations. We are also grateful to the referee for their careful reading and many valuable comments and suggestions. Han is supported by NSFC (Grant No. 12401082). Liu is partially supported by NSFC (Grant No.12301084) and the Natural Science Foundation of Tianjin (Grant No. 22JCYBJC00690, 22JCQNJC01080).

\section{Preliminaries}

We recall basic definitions and fix notation. Throughout, $X$ denotes a metric space with the metric $d$.

A \textbf{geodesic} in $X$ is an isometric embedding of a finite or infinite interval of $\mathbb{R}$ into $X$; we identify it with its image.
$X$ is a \textbf{geodesic metric space} if any two points are joined by a geodesic, and \textbf{proper} if all closed balls are compact.

For $A\subseteq X$ and $r\ge0$, set $\mathcal{N}_r(A):=\{x\in X:d(x,A)\le r\}$ where $d(x,A)=\inf_{a\in A}d(x,a)$. The \textbf{Hausdorff distance} is
\[
d_{\mathcal{H}}(A_1,A_2):=\inf\{r\ge0:A_1\subseteq\mathcal{N}_r(A_2)\text{ and }A_2\subseteq\mathcal{N}_r(A_1)\}.
\]

\begin{Definition}[Quasi-isometric Embedding and Quasi-isometry]
A map $f:(X,d_X)\to(Y,d_Y)$ is a \textbf{$(K,C)$-quasi-isometric embedding} ($K\ge1,C\ge0$) if for all $x,x'\in X$,
\[
\frac{1}{K}d_X(x,x')-C\le d_Y(f(x),f(x'))\le Kd_X(x,x')+C.
\]
If moreover $\mathcal{N}_C(f(X))=Y$, then $f$ is a \textbf{$(K,C)$-quasi-isometry}.
\end{Definition}

\begin{Definition}\label{def:morse}
A geodesic $\alpha\subseteq X$ is \textbf{Morse} if there exists $N:\mathbb{R}_{\ge1}\times\mathbb{R}_{\ge0}\to\mathbb{R}_{\ge0}$ such that every $(K,C)$-quasi-geodesic with endpoints on $\alpha$ lies in $\mathcal{N}_{N(K,C)}(\alpha)$. We call $N$ a \textbf{Morse gauge} and say $\alpha$ is \textbf{$N$-Morse}.
\end{Definition}

For $x,y\in X$ write $[x,y]$ for a geodesic segment joining them. If $\alpha$ is a geodesic and $x,y\in\alpha$, denote by $[x,y]_\alpha$ the oriented subsegment (orientation from $x$ to $y$ which may be consistent or inconsistent with that of $\alpha$ itself). Concatenations are written $[x,y]_\alpha\cup[y,z]$. For a bi-infinite Morse geodesic $\alpha$, let $\alpha_-,\alpha_+\in\partial_*X$ be its endpoints; for a Morse ray $\sigma$, write $\sigma_+$ for its boundary point.
In general, geodesics between two points are not unique. For subsequent treatments of $[x,y]_{\alpha}\cup [y,z]$, the segment $[y,z]$ is typically pre-determined by the context rather than being arbitrary. Furthermore, we consider concatenations of geodesic segments and geodesic rays, thus allowing the two endpoints of the resulting concatenated path to lie on the Morse boundary $\partial_* X$. 

We state below several elementary results on Morse geodesics that will be invoked in the subsequent arguments.

\begin{Lemma}\label{Lemma:SubpathMorse}\cite{Liu21}
For any Morse gauge $N$, there exists a Morse gauge $N'$(depending only on $N$) such that every subpath of an $N$-Morse geodesic is $N'$-Morse.
\end{Lemma}

\begin{Lemma}\label{30QI}
Let $p,q,r\in X\cup\partial_*X$, $o\in X$, and suppose $\alpha_1=[r,p]$ is $N$-Morse. Take $o_1\in\alpha_1$ with $d(o,o_1)\le D$ and let $o_2\in\alpha_2=[p,q]$ be a point closest to $o$. Then:
\begin{enumerate}
\item The concatenation$[o,o_2]\cup[o_2,p]_{\alpha_2}$ is a $(3,0)$-quasi-geodesic.
\item There exists $x\in\alpha_1$ and a constant $C=C(N,D)$ such that $d(x,o_2)\le C$..
\end{enumerate}
\end{Lemma}

    \begin{proof}
    We first verify that the concatenation $[o, o_2]\cup[o_2, p]_{\alpha_2}$ constitutes a $(3, 0)$-quasi-geodesic. Since $[o, o_2]$ and $[o_2, p]_{\alpha_2}$ are geodesics, it suffices to check the quasi-geodesic inequality for arbitrary points $y_1\in [o, o_2]$ and $y_2\in [o_2, p]_{\alpha_2}$. 
    Since $o_2\in\alpha_2$ is the closest point to $o$, it follows that $d(o_2, y_1)\le d(y_1, y_2)$. 
    Applying the triangle inequality, we obtain 
    \[d(y_1, y_2)\le d(y_1, o_2)+d(o_2, y_2)\le d(y_1, y_2)+(d(y_2, y_1)+d(y_1, o_2))\le 3d(y_1, y_2).\] 
    This completes the proof of the inequality. 
    
    Next, we consider two cases based on whether $p\in X$ or $p\in\partial_*X$. Suppose that the point $p\in X$. Define the concatenation $\eta_1=[o_1, o]\cup[o, o_2]\cup[o_2, p]_{\alpha_2}$. It is a $(3, D)$-quasi-geodesic. Since the geodesic $\alpha_1$ is $N$-Morse, then 
    $o_2\in \eta_1\subset \mathcal{N}_{N(3,D)}(\alpha_1),$
    In particular, there exists a point $x\in \alpha_1$ such that $d(x, o_2)\le N(3, D).$

        \begin{figure}[!ht]\label{30}
    \centering
\labellist
  \pinlabel $o$ at 9 44
   \pinlabel $o_1$ at 10 29
    \pinlabel $r$ at -5 30
   \pinlabel $q$ at 50 110
   \pinlabel $p$ at 92 0
   \pinlabel $o_2$ at 65 60
   \pinlabel $x$ at 48 29
   \pinlabel $p_1$ at 69 18
    \pinlabel $p_2$ at 85 27
    \pinlabel $\alpha_2$ at 78 35
    \pinlabel $\alpha_1$ at 56 20

   \endlabellist 
    \includegraphics[width=0.3\linewidth]{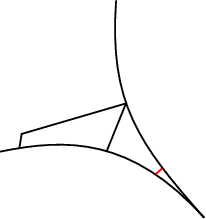}
    \caption{Caption}
    \label{fig1}
\end{figure}
    
    Suppose that the point $p\in \partial_*X$. By assumption, geodesics $\alpha_1$ and $\alpha_2$ are asymptotic to the same point $p$ on the boundary. From the proof of Proposition~2.4 in \cite{Cordes17}, 
    there exists a constant $D_0>0$ depending only on $N$ such that there exist $p_1\in [o_1, p]_{\alpha_1}$ and $p_2\in [o_2, p]_{\alpha_2}$ with $d(p_1, p_2)\le D_0$ (see Figure \ref{fig1}) .
    Consider the concatenation $\eta_2=[o_1, o]\cup[o, o_2]\cup[o_2, p_2]_{\alpha_2}\cup[p_2, p_1]$. It forms a $(3, D_0+D)$-quasi-geodesic. 
    Since the geodesic $\alpha_1$ is $N$-Morse, the concatenation $\eta_2$ lies in the $N(3, D_0+D)$-neighborhood of $\alpha_1$. Thus, there exists $x\in \alpha_2$ such that $d(x, o_2)\le N(3, D_0+D)$. 
    Finally, set $C=\max \{N(3, D_0+D), N(3, D)\}$. The constant $C$ depends only on $N$ and $D$.

    \end{proof}

A geodesic triangle is called \textbf{$\delta$-slim} if each side lies in the $\delta$-neighborhood of the union of the other two sides.

The following lemma, which combines Lemma~2.3 and Lemma~2.4 in \cite{CharneyCordesMurray19}, extends the original formulation (given for vertices in $X$ in \cite{Cordes17}) to the setting where vertices may lie on the Morse boundary.  
 
\begin{Lemma}[Slim triangles and Morse triangles]\label{slim-M}
Let $\triangle(p, q, r)$ be a triangle with vertices $p, q, r\in X\cup\partial_*X$. Suppose that any two sides of $\triangle(p,q,r)$ are $N$-Morse. Then there exists a constant $\delta_N$ and a Morse gauge $N'$, both depending only on $N$, such that the third side is $N'$-Morse and the triangle $\triangle(p,q,r)$ is $\delta_N$-slim.
\end{Lemma}
    
For any $K>0$, every geodesic of length at most $K$ is $N$-Morse for some $N=N(K)$. Consequently, when invoking Lemma \ref{slim-M}, we may include triangles where one or more sides have bounded length.

\subsection{Boundaries of hyperbolic spaces}
We briefly recall the construction of the Gromov boundary and visual metrics for it; detailed accounts can be found in \cite{BridsonHaefliger99, Ghys90}.

\begin{Definition}
    Let $(X, d)$ be a metric space and $x, y,z\in X$. The Gromov product of $x$ and $y$ with respect to $z$ is defined by 
    $$(x\cdot y)_z=\frac{1}{2}(d(x,z)+d(y,z)-d(x,y))$$
\end{Definition}

  In the following definition, the space $X$ is not necessarily geodesic.

\begin{Definition}\label{GHyper}
    Let $(X,d)$ be a metric space and $\delta\ge0$ be a constant. We say $X$ is $\delta$-hyperbolic if for all $x,y,z,w\in X$ we have 
    $$(x\cdot y)_w\ge \min\{(x\cdot z)_w, (y\cdot z)_w \}-\delta.$$
\end{Definition}

Let $o\in X$ be a basepoint. A sequence $(x_n)$ in $X$ converges at infinity if $(x_i\cdot x_i)_0\to \infty$ as $i,j\to \infty$. Two convergent sequences $(x_n)$ and $(y_n)$ in $X$ are said to be equivalent if $(x_i\cdot y_j)_o\to \infty$ as $i,j\to \infty$. The equivalence class $(x_n)$ is denoted $\lim x_n$. The sequential boundary $\partial_sX$ is defined to be the set of equivalence classes. In our paper, we will omit the subscript $s$. 
Let $x,y\in \partial X$. We extend the Gromov product to $\partial X$ by :
$$
(x\cdot y)_o=\sup\displaystyle{\liminf_{i,j\to\infty}}(x_i\cdot y_j)_o,
$$
where the supremum is taken over all sequences $(x_i), (y_j)$ in $X$ with $x=\lim x_i,y=\lim y_j.$

\begin{Definition}
  Let $X$ be a hyperbolic space with a basepoint $o$. A metric $d$ on Gromov boundary $\partial X$ is called a visual metric with parameter $\epsilon>0$ if there exist constants $k_1,k_2>0$ such that for all $p,q\in \partial X$,
  $$k_1\exp{(-\epsilon(p\cdot q)_0)}\le d(p,q)\le k_2\exp{(-\epsilon(p\cdot q)_0)}.$$

\end{Definition}
Here is the standard construction of metrics on $\partial X$. Let $X$ be a $\delta$-hyperbolic space with a basepoint $o$. Let $p, q\in \partial X, \epsilon>0$. Let us consider the following:

$$d_{o,\epsilon}(p,q)=\inf\sum_{i=1}^n\exp{(-\epsilon(p_{i-1}\cdot p_i)_0)},$$
 where the infimum is taken over all finite sequences $p=p_1, p_2, \dots, p_n=q$ in $\partial X$,  no bound on $n$.

\begin{Proposition}\label{VisMetric}

    Let $X$ be a $\delta$-hyperbolic space with a basepoint $o$. Let $\epsilon>0$ be a constant satisfying $4\delta\epsilon\le\ln{2}.$ Then $d_{o,\epsilon }$ is a visual metric on $\partial X$.  For all $p,q\in \partial X$, we have 
    $$(3-2\exp{(2\delta\epsilon)})\exp{(-\epsilon(p\cdot q)_o)}\le d_{o,\epsilon}(p,q)\le \exp{(-\epsilon(p\cdot q)_o)}$$
\end{Proposition}

\begin{Remark}
    The set of all metrics of the form $d_{o,\epsilon}$ is denoted by $\mathcal{G}(X)$ on $\partial X$. Two such metrics $(\partial X, d_{o_1,\epsilon_1})$ and $(\partial X, d_{o_2,\epsilon_2})$ are equivalent if there exists a constant $k>0$ such that $$k^{-1}d_{o_1,\epsilon_1}^{\epsilon_2}\le d_{o_2,\epsilon_2}^{\epsilon_1}\le k d_{o_1,\epsilon_1}^{\epsilon_2}.$$
    Proposition~\ref{VisMetric} tells us that if $o_1,o_2\in X$, $\epsilon_1,\epsilon_2>0$ and $4\delta\epsilon_1,4\delta\epsilon_2\le \ln{2}$, then $(\partial X, d_{o_1,\epsilon_1})$ and $(\partial X, d_{o_2, \epsilon_2})$ are equivalent. The topology on $\partial X$ induced by $d_{o,\epsilon}$ does not depend on the choices of basepoint $0$ and visual parameter $\epsilon$.
\end{Remark}

\subsection{The Morse boundary and Metric Morse boundary}
We first define the \textbf{Morse boundary} for a proper geodesic metric space \( X \). Fix a basepoint \( o \in X \). 
As a set, the Morse boundary \( \partial_*X_o \) consists of the equivalence classes of all Morse geodesic rays emanating from the basepoint \( o \). Here, two geodesic rays are deemed equivalent if their Hausdorff distance is bounded.

For a fixed Morse gauge \( N \), we introduce the following subset:
\[
\partial_*^N X_o = \left\{ [\alpha] \mid \text{there exists an } N\text{-Morse geodesic ray } \beta \in [\alpha] \text{ with basepoint } o \right\}.
\]

Topological structures can be endowed on both \( \partial_*^N X_o \) and \( \partial_*X_o \). However, such topological considerations will be omitted in the present framework. For further details, we refer the reader to \cite{Cordes17}.

Now let us focus on the \textbf{Metric Morse boundary}. Fix a Morse gauge \(N\) and $o\in X$, and define \(X_o^{(N)}\) as the set of points \(x\in X\) for which there exists an \(N\)-Morse geodesic segment \([o, x]\) in \(X\).  
This metric space is not necessarily geodesic. Nevertheless, \(X_o^{(N)}\) is a \(32N(3,0)\)-hyperbolic space in the sense of Definition~\ref{GHyper}. We denote \(\partial X^{(N)}_o\) as the sequential boundary of \(X_o^{(N)}\); for detailed background on Metric Morse boundary, we refer the reader to \cite{CordesHume17}.  

It is necessary to distinguish carefully and distinctly between the two boundaries \(\partial X^{(N)}_{o}\) and \(\partial_*^NX_{o}\) associated with a proper geodesic metric space $X$ and a Morse gauge $N$. Each has its own advantages and limitations, and the choice depends on context. On one hand, \(\partial_*^NX_{o}\) is defined via actual geodesic rays, providing clear geometric intuition; however, it is typically non-metrizable. On the other hand, \(\partial X^{(N)}_{o}\) agrees with the Gromov boundary of a hyperbolic space, even though \(X^{(N)}_{o}\) itself may be neither proper nor geodesic. This discrepancy, however, does not present an obstacle. By applying the following theorem and comparing the two boundaries systematically, we can leverage their complementary strengths.

\begin{Theorem}[\cite{CordesHume17,Liu22}]\label{relation}
    For any Morse gauge $N$, there exists a Morse gauge $N'$ depending only on $N$ such that the following holds.
    Let $X$ be a proper geodesic space with basepoint $o$. There exist two natural embeddings:
    \[
    i_N:\partial X^{(N)}_{o}\to\partial_*^{N'}X_{o} \quad \text{and} \quad j_N:\partial_*^{N}X_{o}\to \partial X^{(N')}_{o}.
    \]
\end{Theorem}

\begin{Remark}
    By virtue of Theorem~\ref{relation}, we may identify the points of $i_N(\partial X_o^{(N)})\subset \partial_*^{N'}X_o$ and $\partial X_o^{(N)}$, as well as those of $j_N(\partial_*^NX_o)\subset \partial X_o^{(N')}$ and $\partial_*^NX_o$, without further distinction in the subsequent discussion.
\end{Remark}

Consider the Gromov boundary $\partial X_o^{(N)}$ of $X_o^{(N)}$: for a suitable visual parameter $\epsilon_N$, there exists a visual metric $d_{o,\epsilon_N}\in \mathcal{G}(X_o^{(N)})$. For the sake of convenience, in any geodesic metric space $X$, we fix a visual parameter $\epsilon_N$ for each Morse gauge $N$ and work exclusively with the visual metric $(\partial X_o^{(N)}, d_{o, \epsilon_N})$. 
For two points $p,q\in \partial X_o^{(N)}$, we denote the Gromov product of $p$ and $q$ by $(p\underset{N}{\cdot} q)_{o}$ (instead of the standard $(p\cdot q)$) to explicitly emphasize the dependence on the Morse gauge $N$.  

The following lemma (Lemma~2.16 in \cite{Liu22}) asserts that, up to an acceptable error margin, $(p\underset{N}{\cdot}q)_o$ is approximately equal to $d(o, [p, q])$.

\begin{Lemma}\label{Gp-dis}
    For any Morse gauge $N$, there exists a constant $k_N>0$ such that the following holds.
    Let $X$ be a proper geodesic metric space with basepoint $o\in X$. For any $p, q\in \partial X_o^{(N)}$ and any geodesic $[p, q]$ connecting $p$ and $q$, we have 
    $$\bigl|(p\underset{N}{\cdot}q)_o-d(o,[p,q])\bigr|\le k_N.$$
\end{Lemma}

\begin{Remark}\label{dis}
    Applying Lemma~\ref{Gp-dis} and Proposition~\ref{VisMetric} to the visual metric $(\partial X_o^{(N)}, d_{o, \epsilon_N})$, 
    we directly deduce a constant $D_N\ge 1$ depending only on $N$, such that for any $p, q\in \partial X_o^{(N)}$ and any geodesic $\gamma$ joining $p$ and $q$, 
    the inequality
    \[D_N^{-1}\exp{(-\epsilon_Nd(o, \gamma))}\le d_{o, \epsilon_{N}}(p,q)\le D_N\exp{(-\epsilon_Nd(o, \gamma))}\]
    holds. This estimate will be invoked in subsequent sections.
\end{Remark}

\section{Uniform perfectness and Morse geodesic richness}
\subsection{Center-exhaustive spaces}

Throughout this subsection, we assume that the Morse boundary $\partial_*X$ contains at least three points.

Fix a Morse gauge $N$. For $n\ge 2$ denote by $\partial_*X^{(n,N)}$ the set of all $n$-tuples $(p_1,\dots,p_n)\in(\partial_*X)^n$ of distinct points such that every bi-infinite geodesic $[p_i,p_j]$ ($i\neq j$) is $N$-Morse. 

By Lemma~\ref{slim-M} there exists $\delta_N=\delta_N(N)>0$ such that every triangle with vertices in $\partial_*X^{(3,N)}$ is $\delta_N$-slim.

\begin{Lemma}[Coarse center; \cite{CharneyCordesMurray19}]\label{Lemma:Center}
Let $(p,q,r)\in\partial_*X^{(3,N)}$ and define
\[
E_K(p,q,r):=\bigl\{x\in X\mid x\text{ lies within }K\text{ of all three sides of some ideal triangle } \triangle(p,q,r)\bigr\}.
\]
For any $K\ge\delta_N$, the set $E_K(p,q,r)$ is non‑empty and has diameter at most $L=L(N, K)$.
\end{Lemma}

For a triple $(p,q,r)\in\partial_*X^{(3,N)}$ and $K$ with $E_K(p,q,r)\neq\emptyset$, points of $E_K(p,q,r)$ are called \textbf{$K$-centers} of $(p,q,r)$; when $K$ is clear from the context we simply speak of \textbf{coarse centers}.

\begin{Definition}[Center-exhaustive space]\label{def:center-exh}
A proper geodesic metric space $X$ is said to be \textbf{center-exhaustive space} if there exist a Morse gauge $N$ and a constant $K\ge0$ such that for every point $x \in X$, there exists a triple $(p, q, r) \in \partial_*X^{(3, N)}$ with $x \in E_K(p, q, r)$.
\end{Definition}

\subsection{Uniform Perfectness of Morse boundaries}

There are several definitions for uniformly perfect spaces in the literature \cite{MAR18, Nowak12, Sugawa01}. The definition below follows Mart\'inez-P\'erez and Rodr\'iguez \cite{MAR18}

\begin{Definition}
    A metric space $(X, d)$ is called uniformly perfect if there exists $S> 1$ and $r_0>0$ such that, for every point $x\in X$ and all $r\in(0, r_0]$, there exists a point $x'\in X$ satisfying \[
    \frac{r}{S} < d(x, x') \le r.
    \]
\end{Definition}

Recall that the Morse boundary, when equipped with its direct limit topology, is typically non-metrizable. This intrinsic feature makes it delicate to formulate a suitable notion of uniform perfectness. 
Our definition is as follows.

\begin{Definition}[Uniform perfectness of $\partial_*X$]\label{UP}
Let $X$ be a proper geodesic metric space with $\partial_*X\neq\emptyset$. 
We say $\partial_*X$ is \textbf{uniformly perfect} if for every Morse gauge $N$ 
with $\bigcup_{o\in X}\partial_*^NX_o\neq\emptyset$, there exist a Morse gauge $N'=N'(N)$, constants $S=S(N)>1$ and $r_0=r_0(N)\in(0,1]$ such that:
for any basepoint $o\in X$ with $\partial_*^NX_o\neq\emptyset$, 
any $p\in\partial_*^NX_o$, and any $r\in(0,r_0]$, there exists $q\in\partial_*X$ with
\[
p,q\in\partial X_o^{(N')}\quad\text{and}\quad\frac{r}{S}<d_{o,\epsilon_{N'}}(p,q)\le r.
\]
Correspondingly, let us say $X$ has a uniformly perfect Morse boundary.
\end{Definition}

Since $\partial_*X\neq\emptyset$, there exists at least one basepoint $o\in X$ and one Morse gauge $N$ 
for which an $N$-Morse ray starting at $o$ exists; consequently, the condition 
$\bigcup_{o\in X}\partial_*^NX_o\neq\emptyset$ is non‑vacuous for some $N$.

\begin{Remark}
Alternative metrizable topologies on the Morse boundary have been constructed for certain groups 
\cite{Cashen19, CordesDussauleGekhtman22, qing24}. One could ask whether uniform perfectness 
can be studied with respect to those metrics. However, such metrics are often geometrically opaque 
and ill‑suited for the intrinsic geometric questions we address here; our definition is therefore 
formulated directly in terms of the visual metrics $d_{o,\epsilon_N}$, which have clear geometric meaning.
\end{Remark}

\subsubsection*{Comparison with the hyperbolic case}

For a proper geodesic Gromov hyperbolic space $X$ and a fixed visual parameter $\epsilon>0$, the classical definition reads:

\begin{Definition}[Uniform perfectness of $\partial X$]\label{def:UP-hyp}
Let $X$ be a proper $\delta$-hyperbolic geodesic metric space with $\partial X\neq\emptyset$. 
We say $\partial X$ is \textbf{uniformly perfect} if there exist constants $S>1$ and $r_0\in(0,1]$ 
such that for every $o\in X$, every $p\in\partial X$, and every $r\in(0,r_0]$, 
there exists $q\in\partial X$ with
\[
\frac{r}{S}<d_{o,\epsilon}(p,q)\le r.
\]
\end{Definition}

When $X$ is hyperbolic, Definition~\ref{UP} coincides with Definition~\ref{def:UP-hyp}: 
all geodesics are uniformly Morse, so the Morse gauges $N$ and $N'$ can be chosen uniformly, and the union $\bigcup_{o\in X}\partial_*^NX_o$ equals $\partial X$ for a suitable $N$.

\begin{Remark}
The definition used in \cite{LiangZhou24} for hyperbolic spaces requires the existence of constants $S>1$ and $r_0>0$ that may depend on the basepoint $o$. Our formulation demands $S$ and $r_0$ to be independent of $o$, giving a genuinely \textbf{uniform} condition. 
\end{Remark}

\subsection{Morse boundary rigid spaces}

Consider now the quasi-isometry group of a metric space $X$, denoted by $\mathcal{QI}(X)$.
This group consists of all self-quasi-isometries $f:X\to X$, modulo the equivalence relation $f_1\sim f_2$ if and only if $\mathrm{sup}_{x\in X}d(f_1(x),f_2(x))$ is finite. 
As established in \cite{CharneySultan15, Cordes17}, every self-quasi-isometry of a proper geodesic metric space $X$ induces a homeomorphism $\partial f$ on its Morse boundary. Furthermore, two uniformly close quasi-isometries induce the same homeomorphism on the Morse boundary. 
Consequently, we obtain a natural homomorphism:
\[\partial:\mathcal{QI}(X)\to \operatorname{Homeo}(\partial_*X) 
\] given by $\partial(f)=\partial f$.

Accordingly, we present the following definition of Morse boundary rigidity:

\begin{Definition}[Morse boundary rigid space]\label{MBR}
Let $X$ be a proper geodesic metric space with non-empty Morse boundary $\partial_*X$.
The following are equivalent definitions:
\begin{enumerate}
\item The natural homomorphism $\partial:\mathcal{QI}(X)\to \operatorname{Homeo}(\partial_*X)$ is injective.
\item For every quasi-isometry $f:X \to X$, if $\partial f=\text{id}_{\partial_*X}$, then $\sup_{x\in X} d(x,f(x))<\infty$.
\end{enumerate}
If either condition holds, $X$ is called \textbf{Morse boundary rigid}.
\end{Definition}

The analogous definition can be formulated for sublinearly Morse boundary rigid spaces. The notion of a boundary rigid space for hyperbolic spaces was first studied in \cite{Shchur13, LiangZhou24}. Several characterizations of this property are established in \cite{LiangZhou24}.
The following result establishes a fundamental relationship between these two notions of boundary rigidity.

\begin{Proposition}
    Let $X$ be a proper geodesic metric space with non-empty $\partial_*X$. If $X$ is Morse boundary rigid, then $X$ is sublinearly Morse boundary rigid.
\end{Proposition}

\begin{proof}
From \cite{qing22, qing24}, every point of the Morse boundary belongs to the sublinearly Morse boundary. For any quasi-isometry $f: X\to X$ whose induced homeomorphism on the sublinearly Morse boundary is the identity, then $f$ induces an identity homeomorphism on the Morse boundary. Since $X$ is Morse boundary rigid, then the displacement  $\sup_{x\in X} d(x,f(x))$ is finite. 
\end{proof}

\subsection{Morse geodesically rich spaces}
 Let us first introduce a class of widely studied spaces: all hyperbolic spaces with non-empty Gromov boundaries, as well as all finitely generated groups with non-empty Morse boundaries. 

\begin{Definition}[Uniformly Morse based space]\label{def:UMB}
A proper geodesic metric space \(X\) is called a \textbf{uniformly Morse based space (UMB space)} if there exists a Morse gauge \(N \)  such that for every point \(x \in X\), there exists at least one $N$-Morse geodesic ray \(\gamma_x \colon [0,\infty) \to X\) with \(\gamma_x(0) = x\).

The class of all such spaces is denoted by \(\UMBspace\). When we need to emphasize the specific gauge, we write \(\UMBclass[N]\).
\end{Definition}

The UMB property is invariant under quasi-isometry. It is well-defined for every finitely generated group. Every Gromov-hyperbolic geodesic metric space with non-empty Gromov boundary is a UMB space. 
If $X$ is a cocompact, proper geodesic metric space with $\partial_*X\neq\emptyset$, then $X$ is a UMB space. In particular, every finitely generated group with non-empty Morse boundary is a UMB space.

The notion of \textbf{geodesic richness} for metric spaces was first introduced in \cite{Shchur13}. Subsequently, Liang and Zhou \cite{LiangZhou24} established that for proper Gromov-hyperbolic geodesic metric spaces, the second condition in the original definition follows directly from the first and can therefore be omitted for brevity (see Lemma~16 in \cite{LiangZhou24}). We extend this line of investigation to \textbf{Morse geodesic richness} for arbitrary proper geodesic metric spaces.

\begin{Definition}[Morse geodesically rich space] \label{Def:GeoRich}
Let $N_0$ be a Morse gauge and let $C_0>0$ be a constant. 
A proper geodesic metric space $X$ is called \textbf{$(N_0, C_0)$-Morse geodesically rich} if the following holds:

For every Morse gauge $N$, there exists a constant $C_1 = C_1(N)$ such that for any pair of points $x_1, x_2 \in X$ joined by an $N$-Morse geodesic, one can find a bi-infinite $N_0$-Morse geodesic $\alpha$ satisfying
  \[d(x_1, \alpha)<C_0\quad\mbox {and}\quad
  \left|d(x_1,x_2)-d(x_2,\alpha)\right|<C_1.\]

  The space $X$ is said to be Morse geodesically rich if it is $(N_0, C_0)$-Morse geodesically rich for some Morse gauge $N_0$ and some constant $C_0>0$.
\end{Definition}

Since every geodesic in hyperbolic space is uniformly Morse, applying the above definition to hyperbolic spaces yields the following:

\begin{Definition}[Geodesically rich hyperbolic space] \label{hyGR}
A proper hyperbolic geodesic metric space $X$ is \textbf{geodesically rich} if there exist positive constants $C_0, C_1$ such that
for any pair of points $x_1, x_2\in X$, there exists a bi-infinite geodesic $\alpha$ satisfying
  \[d(x_1, \alpha)<C_0\quad\mbox{and}\quad\left|d(x_1,x_2)-d(x_2,\alpha)\right|<C_1.\]
\end{Definition}

\begin{Remark}
  For proper hyperbolic geodesic metric spaces, our Definition~\ref{hyGR} differs in presentation from Definition~4.1 in \cite{LiangZhou24}, but the two are equivalent.
\end{Remark}


The following lemma asserts that every Morse geodesically rich space is a UMB space.

\begin{Lemma}\label{N-ray}  
Let $X$ be an $(N_0, C_0)$-Morse geodesically rich space for some Morse gauge $N_0$ and constant $C_0$. Then there exists a Morse gauge $N$, depending only on $N_0$ and $C_0$, such that $X\in \mathcal{UMB}(N)$. 
\end{Lemma}

\begin{proof}
Fix an arbitrary point $o\in X$ and let $x$ be any other point in $X$. Since the geodesic segment $[o, x]$ is of finite length, an elementary argument shows that $[o, x]$ is $N'$-Morse for some Morse gauge $N'$. 
    By the $(N_0, C_0)$-Morse geodesic richness of $X$, there exists a bi-infinite $N_0$-Morse geodesic $\alpha$ such that $d(o, \alpha)<C_0$. 
    Choose a point $o_1\in \alpha$ satisfying $d(o, o_1)<C_0$. By Lemma~\ref{Lemma:SubpathMorse}, the geodesic subsegment $[o_1, \alpha_+]_{\alpha}$ is $N_0'$-Morse, where $N_0'$ depends only on $N_0$. 
    Applying  Lemma~\ref{slim-M}, the geodesic $[o, \alpha_+]$ is $N$-Morse, with $N=N(N_0, C_0)$.

\end{proof}

The following lemma will be used repeatedly in later sections.

\begin{Lemma}\label{keybound}
   Let $X$ be an $(N_0, C_0)$-Morse geodesically rich space for some Morse gauge $N_0$ and constant $C_0>0$. Given any Morse gauge $N$, let $x_1, x_2\in X$ be two points connected by an $N$-Morse geodesic, and let $\alpha$ be the bi-infinite Morse geodesic provided by Definition~\ref{Def:GeoRich}.
    Let $x_1'$ and $x_2'$  be the closest-point projections of $x_1$ and $x_2$ onto $\alpha$.
Then there exists a constant $C>0$, depending only on $N_0, C_0$ and $N$, such that 
\[
d(x_1', x_2')< C.
\]
\end{Lemma}

\begin{proof}
By the Morse geodesic richness of $X$ and the definition of closest point projections $x_1', x_2'$ of $x_1, x_2$ onto $\alpha$, we directly obtain
\begin{equation}\label{geobou}
    d(x_1, x_1')<C_0, \left|d(x_1, x_2)-d(x_2, x_2')\right|< C_1,
\end{equation}
 where $C_1$ is a positive constant depending only on $N$.
Since the geodesic $[x_1, x_2]$ is $N$-Morse and $d(x_1, x_1')<C_0$, it follows from Lemma~\ref{slim-M} that there exists a Morse gauge $N_1$ depending only on $N$ and $C_0$ such that the geodesic $[x_2, x_1']$ is $N_1$-Morse. 
Moreover, by Lemma~\ref{30QI}, the concatenation $[x_2, x_2']\cup[x_2', x_1']_{\alpha}$ forms a $(3, 0)$-quasi-geodesic and there exists a point $x_3\in [x_1', x_2]$ satisfying 
    \[d(x_2', x_3)\le N_1(3,0).\]
    Applying the triangle inequality twice, we first derive that 
    \[
\begin{aligned}
\left|d(x_1',x_2)-d(x_2,x_2')-d(x_2',x_1')\right|
&\le \left|d(x_1',x_2')-d(x_1',x_3)\right| + \left|d(x_2',x_2)-d(x_2,x_3)\right| \\
&\le 2N_1(3, 0).
\end{aligned}
\]

Then, combining with the triangle inequality again, we have 
$$\left|d(x_1,x_2)-d(x_2,x_2')-d(x_2', x_1')\right|\le \left|d(x_1,x_2)-d(x_1',x_2)\right|+2N_1(3, 0).$$
Now, substituting the inequalities from Equation~\eqref{geobou}, we find that
\[d(x_1', x_2')<C_0+C_1+2N_1(3, 0).\] 
As $C_1$ and $N_1$ depend only on $N$, we can set $C=C_0+C_1+2N_1(3, 0)$, where $C$ is a constant depending only on $N_0, C_0$ and $N$. This completes the proof.

\end{proof}

\section{Proof of Theorem~\ref{thm:equivalence}}

\subsection{Morse geodesic richness and center-exhaustiveness}
In this subsection, we show that Morse geodesic richness is equivalent to center-exhaustiveness, a result guaranteed by the next two propositions.

\begin{Proposition}\label{MGR-CCF}
   Let $X$ be a proper geodesic metric space. If $X$ is Morse geodesically rich, then $X$ is center-exhaustive. 
\end{Proposition}

\begin{proof}
Suppose $X$ is $(N_0, C_0)$-Morse geodesically rich with respect to some Morse gauge $N_0$ and some constant $C_0$.
For any $a\in X$, by Lemma~\ref{N-ray}, there exists an $N$-Morse geodesic ray $\sigma$ emanating from $a$, where the Morse gauge $N$ depends only on $N_0$ and $C_0$. 
According to Lemma~\ref{Lemma:SubpathMorse}, every geodesic subsegment of $\sigma$ is $N'$-Morse, where $N'=N'(N)$ is a Morse gauge depending only on $N$.
By the definition of Morse geodesical richness of $X$, there exists a constant $C_1$ depending only on $N'$ such that, 
for the point $b\in\sigma$ and $b\neq a$,
there exists an $N_0$-Morse bi-infinite geodesic $\alpha$ satisfying the following conditions: 
\begin{equation}\label{ab}
    d(a, \alpha)<C_0,\quad \left|d(b, \alpha)-d(a, b)\right|< C_1.
\end{equation}
Let $a', b'\in \alpha$ denote the nearest points on $\alpha$ to $a$ and $b$, respectively(see Figure~\ref{fig3}).

Consider the ideal triangle $\triangle(a, a', \sigma_+)$. Since the geodesic segment $[a', a]$ has length at most $C_0$ and the geodesic ray $\sigma$ is $N$-Morse, 
from Lemma~\ref{slim-M} that the ideal triangle $\triangle(a,a', \sigma_+)$ is $\delta_1$-slim, where $\delta_1$ is a positive constant determined exclusively by $N$ and $C_0$. Thus, the constants $\delta_1+C_0$ and $\delta_1+C_1$ depend only on $N_0$ and $C_0$, and are independent of $d\in\sigma$.
Now, choose $b\in \sigma$ such that 
\[d(a, b)> \max\{\delta_1+C_0, \delta_1+C_1\}.\]
We proceed to show that $\sigma_+\notin\{\alpha_-, \alpha_+\}$. Suppose, for contradiction, that $\sigma_+\in\{\alpha_-, \alpha_+\}$.
Then we have
\[
d(b, [a, a'])\ge d(a,b)-C_0>\delta_1.
\]
This inequality implies that, 
\[
d(b, [a', \sigma_+]_{\alpha})\le \delta_1.
\]
It follows that
\[
d(b, b')=d(b, \alpha)\le d(b, [a', \sigma_+]_{\alpha})\le \delta_1.
\]
Substituting this into Equation~\eqref{ab}, 
\[
d(a,b)<C_1+d(b, b')\le C_1+\delta_1.
\] This clearly contradicts our initial choice of $b$. 

By Lemma~\ref{Lemma:SubpathMorse}, the geodesic segments $[a', \alpha_+]_\alpha$, $[\alpha_-, a']_{\alpha}$ are $N_0'$-Morse for some Morse gauge $N_0'$ depending only on $N_0$. 
Lemma~\ref{slim-M} guarantees the existence of a Morse gauge $N_1=N_1(C_0, N_0')$ such that the geodesics $[a, \alpha_+]$ and $[a, \alpha_-]$ are $N_1$-Morse. And the ideal triangles $\triangle(a, a', \alpha_+)$ and $\triangle(a, a', \alpha_-)$ are $\delta$-slim, where $\delta$ is a positive constant depends only on $C_0$ and $N_0'$. 
Similarity, the ideal triangles $\triangle(a, \sigma_+,\alpha_+)$ and $\triangle(a, \sigma_+, \alpha_-)$ are $\delta'$-slim for some positive $\delta'=\delta'(N_1, N)$. The geodesics $[\alpha_-, \sigma_+]$ and $[\alpha_+,\sigma_+]$ are $N_2$-Morse, where the Morse gauge $N_2=(N_1, N')$ depends on $N_1$ and $N'$. Set $\delta_2=\max\{\delta, \delta'\}$. Since $N_1, N, N'$ and $N_0'$ all ultimately depend on $N_0$ and $C_0$, the constant $\delta_2$ and the Morse gauge $N_2$ depend only on $N_0$ and $C_0$.

\begin{figure}[!ht]
    \centering
\labellist
\small{
  \pinlabel $a$ at 33 51
  \pinlabel $a'$ at 21 47
  \pinlabel $b'$ at 22 38
  \pinlabel $\alpha_+$ at 6 -1
  \pinlabel $b_1$ at 99 37
  \pinlabel $b$ at 92 45
  \pinlabel $b_1'$ at 38 18
  \pinlabel $\alpha_-$ at -5 80
  \pinlabel $\alpha$ at 18 56
  \pinlabel $\sigma_+$ at 134 48
  \pinlabel $\sigma$ at 66 51
  }
   \endlabellist 
    \includegraphics[width=0.4\linewidth]{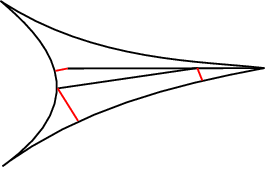}
    \caption{Caption}
    \label{fig3}
\end{figure}

Assume that $b\in \sigma$ is a point satisfying 
\[d(a,b)> \max\{\delta_1+C_0, \delta_1+C_1, 2\delta_2+C_0+C_1\}.\]
By applying the slim triangle property of the ideal triangle $\triangle(a, a', \alpha_+)$ and using the fact that $d(a, a')<C_0$, we derive the inequality
$$d(b, [a', \alpha_+]_{\alpha})< d(b, [a, \alpha_+])+\delta_2+C_0.$$
Combining this with Equation~\eqref{ab}, we obtain the following chain of inequalities:

\begin{align*}
d(a, b)
&< d(b, b')+C_1 \\
&\le d\bigl(b, [a', \alpha_+]_{\alpha}\bigr)+C_1 \\
&< d\bigl(b, [a, \alpha_+]\bigr)+\delta_2+C_0+C_1.
\end{align*}
It follows immediately that $d(b, [a, \alpha_+])>\delta_2$. 
The ideal triangle $\triangle(a, \alpha_+, \sigma_+)$ is $\delta_2$-slim. It follows that there exists a point $b_1\in[\sigma_+, \alpha_+]$ such that 
\[d(b, b_1)\le\delta_2.\]
 
Consider the concatenation $[b_1, b]\cup[b, b']\cup[b', \alpha_+]_{\alpha}$. The geodesic $[\alpha_+, \sigma_+]$ is $N_2$-Morse.
By Lemma~\ref{30QI}, 
there exists a positive constant $C_2$ depending only on $N_2$ and $\delta_2$ such that there is a point $b_1'\in [\alpha_+, \sigma_+]$ satisfying 
\[d(b', b_1')\le C_2.\]
Similarity, there exists a point $b_2'\in [\alpha_-,\sigma_+]$ such that 
\[d(b', b'_2)\le C_2.\]
In addition, Lemma~\ref{keybound} guarantees the existence of a constant $C$ depending only on $N_0, C_0$ and $N'$ such that 
\[d(a', b')<C.\] 
Let $K=C_0+C+C_2$ and $N_3=\max\{N_2, N_0\}$. It follows that the point $a\in E_K(\sigma_+, \alpha_-, \alpha_+)$ with $(\sigma_+, \alpha_-, \alpha_+)\in \partial_*X^{(3, N_3)}$. Notably, both the constant $K$ and the Morse gauge $N_3$ depend only on $N_0$ and $C_0$.
   
\end{proof}

\begin{Proposition}\label{cce-mgr}
     Let $X$ be a proper geodesic metric space. If $X$ is center-exhaustive, then $X$ is Morse geodesically rich. 
\end{Proposition}

\begin{proof}

Let $N$ be a Morse gauge and suppose $a,b\in X$ are such that there exists an $N$-Morse geodesic $[a,b]$. 
Since $X$ is a center-exhaustive space, there exist a Morse gauge $N_0$ and a constant $K\ge0$ such that $a$ is a $K$-center of $(p,q,r)\in \partial_*X^{(3, N_0)}$, where the sides $\alpha_1=[p,q],\alpha_2=[q,r]$ and $\alpha_3=[p,r]$. 
The positive constant $\delta_{N_0}$ depends only on $N_0$ such that the triangle $\triangle(p,q,r)$ is $\delta_{N_0}$-slim.

\begin{figure}[!ht]
    \centering
\labellist
\small
{
  \pinlabel $p$ at 0 27
  \pinlabel $a$ at 135 80
  \pinlabel $a_2$ at 125 35
  \pinlabel $r$ at 270 10
  \pinlabel $q$ at 120 207
  \pinlabel $\alpha_1$ at 57 80
  \pinlabel $\alpha_3$ at 172 88
  \pinlabel $\alpha_2$ at 168 28
  \pinlabel $b$ at 79 170
  \pinlabel $b_1$ at 124 153
  \pinlabel $b_2$ at 98 35
  \pinlabel $a_1$ at 102 100
}
   \endlabellist 
    \includegraphics[width=0.4\linewidth]{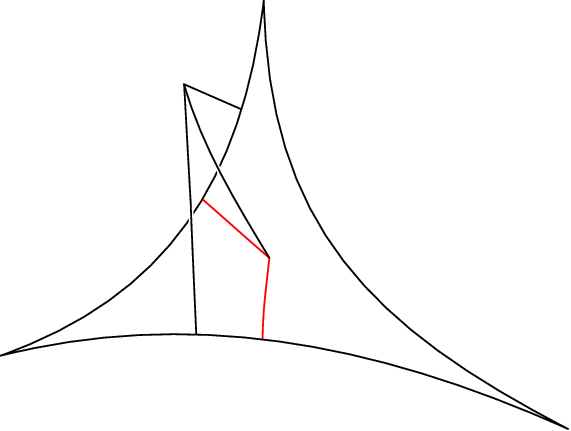}
    \caption{Caption}
    \label{fig4}
\end{figure}

  Let $b_1$ be the closest-point projection of $b$ onto $\triangle(p,q,r)$. Without loss of generality, assume $b_1\in\alpha_1\cap\mathcal{N}_{\delta_{N_0}}(\alpha_3)$. Let $b_2\in\alpha_2$ be the closest-point projection of $b$ onto $\alpha_3$. See Figure~\ref{fig4}.
  
  \vspace{1em}
  \textbf{Claim:} There exists a constant $D=D(N_0,K,N)$ such that either
\[
  {d(a, b_1)\le D}\quad\text{or}\quad
  {d(a, b_2)\le D}.
\]
  
We postpone the proof of this claim to the end of the argument. 
With the claim established,
set $C_0=K+1, C_1=D+1$; 
these are constants depending only on $K, N_0$, and $N$.
If $d(a, b_1)\le D$, then we have 
 $$d(a, \alpha_1)< C_0 \quad\text{and}\quad\left|d(a, b)-d(b,\alpha_1)\right|\le d(a, b_1)< C_1.$$ 
 In this case, we select $\alpha_1$ as the desired bi-infinite $N_0$-Morse geodesic. 
Similarity, if $d(a, b_2)\le D$, we select $\alpha_2$ as the desired  bi-infinite $N_0$-Morse geodesic. This completes the proof that $X$ is an $(N_0, C_0)$-Morse geodesically rich space. 

\textbf{Proof of the claim: }

Since $a$ is a $K$-center of $\triangle(p,q,r)$, we may choose points $a_1\in\alpha_1, a_2\in\alpha_2$ such that $d(a,a_1)\le K$ and $
d(a,a_2)\leq K$. By Lemma~\ref{30QI}, the concatenations $\eta_1=[a_1,b_1]_{\alpha_1}\cup[b_1,b]$ and $\eta_2=[a_2,b_2]_{\alpha_2}\cup[b_2,b]$ are $(3,0)$-quasi-geodesics.
Thus, the concatenations $[a, a_1]\cup\eta_1$ and $[a, a_2]\cup\eta_2$ are continuous $(3,K)$-quasi-geodesics whose endpoints are $a$ and $b$. Since $[a, b]$ is $N$-Morse, by Lemma~2.1 in \cite{Cordes17}, 
then there exists a constant $2N(3,K)$ depending only on $K$ and $N$ such that
\[d_{\mathcal{H}}([a, b], [a, a_1]\cup\eta_1)\le 2N(3,K)\quad\text{ and}\quad d_{\mathcal{H}}([a, b], [a, a_2]\cup\eta_2)\le 2N(3,K).\]
Note that $d(a, a_1)\le K$ and $d(a,a_2)\le K$. This implies that 
\[
d_{\mathcal{H}}(\eta_1, \eta_2)\le 4N(3,K)+2K
\]
In particular, there exists a point $c\in\eta_1$ such that \[d(b_2,c)\leq D_1,\]
where $D_1= 4N(3,K)+2K$. We now consider two cases.

\textbf{Case I: }$\boldsymbol{c\in [b_1,b]}$

Note that $b_1$ denotes the closest point projection of $b$ onto $\triangle(p,q,r)$ and $c\in[b, b_1]$. It follows that
  $b_1$ is the closest point projection of $c$ onto $\triangle(p,q,r)$. Thus, 
\[d(c,b_1)\leq d(c, b_2)\leq D_1,\] which further implies that \[d(b_1,b_2)\leq d(b_1,c)+d(b_2,c)\leq 2D_1.\] 
  Therefore, points $a, b_1\in E_{K_1}(p,q,r)$, where $K_1=\max\{K, 2D_1, \delta_{N_0}\}\ge \delta_{N_0}$. By Lemma~\ref{Lemma:Center}, there exists a positive constant $L_1$ depending only on $N_0, D_1$ and $K$ such that
  \[d(a, b_1)\le L_1.\]

 \textbf{Case II: } $\boldsymbol{c\in[a_1,b_1]_{\alpha_1}}$ 

Given that $d(b_1, \alpha_3)\le \delta_{N_0}, d(a_1, \alpha_3)\le d(a_1, a)+d(a, \alpha_3)\le 2K$, with the $N_0$-Morse property of geodesic $\alpha_3$, a standard argument guarantees the existence of a constant $D_2=N_0(1, \delta_{N_0}+2K)$ such that 
\[c\in[b_1, a_1]_{\alpha_1}\subset \mathcal{N}_{D_2}(\alpha_3).\]
Consequently, the distance between $b_2$ and $\alpha_3$ satisfies the inequality
\[d(b_2, \alpha_3)\le d(b_2, c)+d(c, \alpha_3)\le D_1+D_2.\]
Also, the distance $d(b_2, \alpha_1)\le d(b_2, c)\le D_1$. This shows that points $a, b_2\in E_{K_2}(p,q,r)$, where $K_2=\max\{K, D_1+D_2, \delta_{N_0}\}\ge \delta_{N_0}$. Again by Lemma~\ref{Lemma:Center}, there exists a constant $L_2$ depending only on $N_0, D_1, D_2$ and $K$ such that 
\[
d(a, b_2)\le L_2.
\]

Set $D=\max\{L_1, L_2\}$. Then $D$ depends only on $N_0, K$ and $N$. Hence, the desired claim follows immediately.

\end{proof}

 The proof of Theorem~\ref{thm:equivalence} will be completed using the three propositions below, developed in the subsequent two subsections.

\subsection{Morse geodesic richness implies uniform perfectness }

\begin{Proposition}\label{MGR-UP}
If $X$ is Morse geodesically rich, then the Morse boundary $\partial_*X$ is uniformly perfect.    
\end{Proposition}

\begin{proof}
Suppose $X$ is an $(N_0, C_0)$-Morse geodesically rich space for some Morse gauge $N_0$ and some constant $C_0$.
   Let $N$ be any Morse gauge. Fix a basepoint $o\in X$ and suppose that $\partial_*^NX_o$ is non-empty. For any point $p\in\partial_*^NX_o$, choose a geodesic ray $\sigma$ emanating from $o$ to $p$.
    
   For any point $x\in \sigma$ and $x\neq o$, Lemma~\ref{Lemma:SubpathMorse} implies that the geodesic segment $[o,x]_{\sigma}$ and the geodesic ray $[x, p]_{\sigma}$ are $N_1$-Morse, where $N_1= N_1(N)$ is a Morse gauge depending only on $N$.
For the $N_1$-Morse geodesic $[o, x]$, by the Morse geodesic richness of $X$,
there exists a constant $C_1>0$ depending only on $N_1$ and a bi-infinite $N_0$-Morse geodesic $\alpha$ satisfying 

\begin{equation}\label{C0C1}
     d(x,\alpha)<C_0, \left|d(o, \alpha)-d(o,x)\right|<C_1.
\end{equation}

   Let $x'$ and $o'$ denote the closest point projections of $x$ and $o$ onto $\alpha$, respectively. 
  By Lemma~\ref{keybound}, there exists a positive constant $C_2$ depending only on $N_0, C_0, N_1$ such that 
  \begin{equation}\label{C2}
    d(o', x')<C_2.
  \end{equation}

\textbf{Claim A:}
We have $p,\alpha_-,\alpha_+\in \partial X_o^{(N')}$ for some Morse gauge $N'$ depending only on $N$. 

Let us prove this claim. 
The geodesic segment $[o,x]_{\sigma}$ is $N_1$-Morse and $d(x,x')<C_0$; 
hence, by Lemma~\ref{slim-M}, the geodesic segment $[o, x']$ is $N_2$-Morse, where $N_2=N_2(N_1, C_0)$. 
Note that Lemma~\ref{Lemma:SubpathMorse} applied again implies that the geodesics $[x',\alpha_+]_\alpha$ and $[x',\alpha_-]_\alpha$ are $N_0'$-Morse with $N_0' = N_0'(N_0)$. 
Lemma~\ref{slim-M} ensures that the geodesics $[o,\alpha_+]$ and $[o,\alpha_-]$ are $N_3$-Morse, where $N_3=N_3(N_0', N_2)$. 
By Theorem~\ref{relation}, it follows that $p, \alpha_-,\alpha_+\in \partial X_o^{(N')}$ for some $N'=N'(N_3, N)$. Since $N_3$ ultimately depends only on $N$, the Morse gauge $N'$ depends only on $N$. Hence, the claim holds.

The geodesic $\alpha$ depends on the choice of $x\in \sigma$. 
We now distinguish two cases for further analysis.

\textbf{Case I}: $\boldsymbol{p\in \{\alpha_-, \alpha_+\}}$  

Without loss of generality, we assume that $p=\alpha_-$. By Remark~\ref{dis}, there exists a constant $D_{N'}>0$ such that
\[
D_{N'}^{-1}\exp\left(-\epsilon_{N'}d(o,\alpha)\right) < d_{o, \epsilon_{N'}}(\alpha_-, \alpha_+) < D_{N'}\exp\left(-\epsilon_{N'}d(o,\alpha)\right).
\]

Since $\left|d(o, \alpha)-d(o, x)\right|<C_1$, it follows that
\[
D_1^{-1}\exp\left(-\epsilon_{N'}d(o,x)\right) < d_{o, \epsilon_{N'}}(p, \alpha_+) < D_1\exp\left(-\epsilon_{N'}d(o,x)\right),
\]
where $D_1=D_{N'}\exp\left(C_1\epsilon_{N'}\right)$.

\textbf{Case II:} $\boldsymbol{ p\notin \{\alpha_-, \alpha_+\}}$ 

\begin{figure}[!ht]
    \centering
\labellist
\small{
  \pinlabel $o$ at 0 93
  \pinlabel $x$ at 30 93
  \pinlabel $p$ at 110 90
  \pinlabel $\alpha_-$ at 105 133
  \pinlabel $\alpha$ at 60 115
  \pinlabel $x_1$ at 26 60
  \pinlabel $o'$ at 41 70
  \pinlabel $x'$ at 52 85
  \pinlabel $x''$ at 66 67
  \pinlabel $\alpha_+$ at 61 10
  }
   \endlabellist 
    \includegraphics[width=0.25\linewidth]{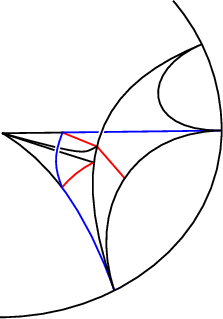}
    \caption{Picture of Case II in Theorem~\ref{MGR-UP}}
    \label{fig:4}
\end{figure}

Let us consider the ideal triangle $\triangle(p, \alpha_-, \alpha_+)$. As we shall show in the following argument, this ideal triangle is $\delta$-slim for some constant $\delta = \delta(N)$ depending only on $N$.

Since $[o, \alpha_+]$ is $N_3$-Morse and the geodesic $\sigma$ is $N$-Morse, Lemma~\ref{slim-M} implies that the geodesic $[p, \alpha_+]$ is $N_4$-Morse, where $N_4 = N_4(N_3, N)$. With the $N_0$-Morse geodesic $\alpha$, by Lemma~\ref{slim-M}, it follows that the ideal triangle $\triangle(p, \alpha_-, \alpha_+)$ is $\delta$-slim, with $\delta$ depending only on $N_4$ and $N_0$. Recalling that $N_3$ itself depends only on $N$, we conclude that $\delta$ depends only on $N$.

Thus, for the point $x'\in \alpha$, without loss of generality, there exists a point $x''\in [p, \alpha_+]$ such that $d(x',x'')\le\delta$ (see Figure~\ref{fig:4}). 

We first derive the bound for $d(o, [p,\alpha_+])$:
\[
d(o, [p,\alpha_+])\le d(o,x'')\le d(o, x)+d(x, x')+d(x', x'')< d(o, x)+C_0+\delta.
\]
By Remark~\ref{dis}, we have
\begin{equation}\label{D_3}
    D_2\exp\left(-\epsilon_{N'}d(o, x)\right)<d_{o, \epsilon_{N'}}(p, \alpha_+),
\end{equation}
where $D_2=D_{N'}^{-1}\exp\left(-\epsilon_{N'}(C_0+\delta)\right)>0$. 

Since the geodesic $[o, \alpha_+]$ is $N_3$-Morse, Lemma~\ref{30QI} guarantees the existence of a point $x_1\in [o, \alpha_+]$ and a constant $C_3=C_3(N_3)$ such that 
\[
d(o', x_1)\le C_3.
\] 
Thus, set $D_3=C_0+C_2+C_3$ and by Equation~\eqref{C0C1} and Equation~\eqref{C2} we obtain
\begin{equation}\label{D}
    d(x, x_1)\le d(x,x' )+d(x', o')+d(o', x_1)<D_3,
\end{equation}
where the constant $D_3$ depends only on $N$.

\textbf{Claim B:}
There exists a constant $D_4 = D_{N'}\exp\left(\epsilon_{N'}(D_3+2\delta')\right)$, where $\delta'$ is the constant defined below, such that
\begin{equation}\label{D_4}
    d_{o,\epsilon_{N'}}(p, \alpha_+)< D_4\exp\left(-\epsilon_{N'}d(o, x)\right).
\end{equation}

We proceed with the proof of the claim as follows. First, select a point $o_1\in [p, \alpha_+]$ such that $d(o, [p, \alpha_+])=d(o, o_1)$.  
Given that the geodesic segment $[x, p]$ is $N_1$-Morse and $[p, \alpha_+]$ is $N_4$-Morse, Lemma~\ref{slim-M} guarantees that the triangle $\triangle(x, p, \alpha_+)$ is $\delta_1$-slim, where the constant $\delta_1 = \delta_1(N_1, N_4)$. 
 Next, Lemma~\ref{Lemma:SubpathMorse} implies that the geodesic $[x_1, \alpha_+]_{\alpha}$ is $N_3'$-Morse for some Morse gauge $N_3'=N_3'(N_3)$. Again applying Lemma~\ref{slim-M}, the geodesic $[x, \alpha_+]$ is $N_4=N_4(C_0, N_0')$-Morse for some Morse gauge $N_4$ and the triangle $\triangle(x, x_1, \alpha_+)$ is $\delta_2$-slim, with $\delta_2=\delta_2(N_4,N_3')$. Finally, it is straightforward to verify that the constant $\delta'=\max\{\delta_1,\delta_2\}$ depends only on the relevant Morse gauges, which themselves depend only on $N$.

Let $\sigma$ denote the geodesic concatenation $[p, x]_\sigma\cup[x, x_1]\cup[x_1, \alpha_+]$. Applying the slim triangle property twice, we have
\[
d(o_1, \sigma)\le 2\delta'.
\] 
By the triangle inequality, it is straightforward to see that,
\[
d(o, \sigma)\ge d(o, x)-d(x, x_1).
\]
Combining this with Equation~\eqref{D}, we obtain
\[
d(o, \sigma)> d(o, x)-D_3.
\]

To establish Equation~\eqref{D_4}, we first need to derive a lower bound for the distance from $o$ to the geodesic $[p, \gamma_+]$. Applying the triangle inequality once more, we have  
\[
d(o, o_1)\ge d(o, \sigma)-d(o_1, \sigma)> (d(o, x)-D_3)-2\delta'.
\]
That is, \[d(o, [p, \alpha_+])>d(o, x)-(D_3+2\delta')\]

Thus, by setting $D_4=D_{N'}\exp\left(\epsilon_{N'}(D_3+2\delta')\right)$, the desired Equation~\eqref{D_4} follows directly from Remark \ref{dis}, which completes the proof of Claim B.

Combining Equation~\eqref{D_3} and Equation~\eqref{D_4}, we obtain
\[
D_2\exp\left(-\epsilon_{N'}d(o, x)\right)<d_{o, \epsilon_{N'}}(p, \alpha_+)< D_4\exp\left(-\epsilon_{N'}d(o, x)\right).
\]

Let $m=\min\{D_1^{-1}, D_2\}>0$ and $M=\max\{D_1,D_4\}$ be constants. Note that both $m$ and $M$ depend only on $N$ and the Morse geodesic richness of the space $X$. For any Morse gauge $N$ and any basepoint $o\in X$ with $\partial_*^NX_o\neq\emptyset$, there exist a Morse gauge $N'$ (as in Claim A) and constants $\displaystyle{S=\frac{M}{m}>1}$ and $\displaystyle{r_0=\min\{M, 1\}}\in(0, 1]$ such that the following holds:

For any point $p\in\partial_*^NX_o$ and every $r\in (0, r_0]$, choose a point $x\in \sigma$ satisfying 
\[
d(o, x)=\frac{1}{\epsilon_{N'}}\ln\left(\frac{M}{r}\right),
\]
where $\sigma$ is the geodesic ray from $o$ to $p$ as defined previously. As discussed above, the geodesic $\alpha$ depends on the choice of $x\in \sigma$. In both Case I and Case II, we can always select a point $q\in \{\alpha_-, \alpha_+\}$ such that $p, q\in \partial X^{(N')}_o$ and
\[
\frac{r}{S}<d_{o,\epsilon_{N'}}(p,q)<r.
\]
This completes the proof of the desired result.

\end{proof}

\subsection{Uniformly perfectness implies center-exhaustiveness}

\begin{Proposition}
    Let $X\in\mathcal{UMB}(N)$ for some Morse gauge $N$. If its Morse boundary $\partial_*X$ is uniformly perfect, then $X$ is center-exhaustive.  
\end{Proposition}

\begin{proof}

For any point $o\in X$, since $X\in\mathcal{UMB}(N)$, we take a point $p\in \partial_*^{N}X_o$.

By virtue of the uniform perfectness of the Morse boundary $\partial_* X$, there exist a Morse gauge $N'=N'(N)$ and constants $S=S(N)> 1$ and $r_0=r_0(N)\in (0, 1]$ such that for any $r\in (0, r_0]$, one can find a point $q\in \partial_* X$ satisfying
\[
p, q\in \partial X^{(N')}_o \quad \text{and} \quad \frac{r}{S}<d_{o,\epsilon_{N'}}(p,q)\le r.
\]

In particular, setting $r=r_0$ and $\displaystyle r=\frac{r_0}{2S}$ respectively, we can choose two points $q_1, q_2\in \partial X^{(N')}_o$ that satisfy
\[
\frac{r_0}{S}<d_{o,\epsilon_{N'}}(p,q_1)\le r_0\quad\text{and}\quad
\frac{r_0}{2S^2}<d_{o,\epsilon_{N'}}(p,q_2)\le \frac{r_0}{2S}.
\]
It then follows that
\[
\frac{r_0}{2S}<d_{o,\epsilon_{N'}}(q_1,q_2).
\]

Moreover, all the distances $d_{o,\epsilon_{N'}}(p,q_2)$, $d_{o,\epsilon_{N'}}(p,q_1)$ and $d_{o,\epsilon_{N'}}(q_1,q_2)$ admit a uniform lower bound of $\displaystyle\frac{r_0}{2S^2}$.
According to Remark~\ref{dis}, the distances $d(o, [p, q_1])$, $d(o, [p, q_2])$ and $d(o, [q_1, q_2])$ are uniformly bounded above by
\[
K=\frac{1}{\epsilon_{N'}}\ln\left(\frac{2S^2D_{N'}}{r_0}\right)>0.
\]
Since $q_1, q_2\in \partial X_o^{(N')}$, Theorem~\ref{relation} ensures that the geodesics $[o, q_1]$ and $[o, q_2]$ are $N_1$-Morse, where $N_1$ depends only on $N'$. 
Combining this with the fact that $[o, p]$ is an $N$-Morse geodesic, 
Lemma~\ref{slim-M} implies the existence of a Morse gauge $N_0$ (depending only on $N'$ and $N_1$) such that the geodesics $[p, q_1]$, $[p, q_2]$ and $[q_1, q_2]$ are all $N_0$-Morse. 
Hence, we have $(p, q_1, q_2)\in \partial_*X^{(3, N_0)}$.

To summarize, for any $o\in X$, the point $o\in E_{K}(p, q_1, q_2)$ with $(p, q_1,q_2)\in \partial_*X^{(3, N_0)}$. As both $K$ and $N_0$ depend exclusively on $N$, we conclude that the space $X$ is center-exhaustive.

\end{proof}

\begin{Proposition}
Suppose that $X$ is a proper geodesic metric space which is center-exhaustive, then $X$ is Morse boundary rigid.    
\end{Proposition}

\begin{proof}
Suppose that $f: X\to X$ is a quasi-isometry such that the induced homeomorphism $\partial f:\partial_*X\to\partial_*X$ is the identity. 
Since $X$ is center-exhaustive, there exists a Morse gauge $N$ and a constant $K\ge0$ such that every $x\in X$ is a $K$-center of some triple $(p,q,r)\in\partial_*X^{(3, N)}$. That is, $x$ lies within $K$ of all three sides of some ideal triangle $\triangle(p, q, r)$. Under the quasi-isometry \( f \), the images of the sides of \( \triangle(p,q,r) \) are quasi-geodesics that lie within a uniformly bounded Hausdorff distance $D$ from the original sides, where $D$ depends only on $f$ and $N$.
Note that \( f(x) \) is at a uniformly bounded distance $D_1$ from these three image quasi-geodesics, where $D_1$ depends only on $f$ and $K$. Hence  \( f(x) \) is at a uniformly bounded distance $D_1+D$ from the three sides of \( \triangle(p,q,r) \). Since \( x \) is a \( K \)-center of \( \triangle(p,q,r) \), this implies 
\[
x, f(x)\in E_{K'}(p, q, r), \quad\text{where}\quad K'=\max\{K, D_1+D, \delta_N\}.
\]
By Lemma~\ref{Lemma:Center}, the distance $d(x, f(x))\le L$, where $L=L(K, D_1+D, N)$. 
Consequently, the displacement \( \sup_{x \in X} d(x, f(x)) \) is finite, so \( X \) is Morse boundary rigid. 
\end{proof}

\section{Quasi-isometry invariants}
In this section, we will prove that many properties defined in the above section are preserved under quasi-isometries.

\begin{Proposition}
     Let $(X, d_X)$ and $(Y, d_Y)$ be proper geodesic metric spaces. If $f: X \to Y$ is a quasi-isometry and $X$ is center-exhaustive, then $Y$ is center-exhaustive.
\end{Proposition}

\begin{proof}
   Since $f(X)$ is $C$-dense in $Y$ for some constant $C$ depending only on $f$, for every $y\in Y$, there exists a point $x\in X$ such that $d_Y(f(x), y)\le C$.
   As $X$ is center-exhaustive, there exist a Morse gauge $N$ and a constant $K\ge 0$ such that every $x\in X$ is a $K$-center of some $(p,q,r)\in\partial_*X^{(3,N)}$. Under the quasi-isometry $f$, the images of bi-infinite $N$-Morse geodesics are quasi-geodesics that lie within uniformly finite Hausdorff distance $K'$ from some bi-infinite $N'$-Morse geodesic, where the Morse gauge $N'$ and constant $K'$ depend only on $N$ and $f$. It follows that $y$ is $K'+C$-center of $(\partial f(p), \partial f(q), \partial f(r))\in \partial_*Y^{(3, N')}$. Here, both $K'+C$ and $N'$ depend only on $f$ and $X$. Therefore, $Y$ is center-exhaustive.
   
\end{proof}

From Proposition~\ref{MGR-CCF} and Proposition~\ref{cce-mgr}, a proper geodesic metric space $X$ is Morse geodesically rich if and only if it is center-exhaustive. Since the center-exhaustive property is a quasi-isometry invariant, it follows that Morse geodesic richness is also a quasi-isometry invariant.
\begin{Proposition}
     Let $(X, d_X)$ and $(Y, d_Y)$ be proper geodesic metric spaces. If $f: X \to Y$ is a quasi-isometry and $X$ is Morse geodesically rich, then $Y$ is Morse geodesically rich.
\end{Proposition}

For hyperbolic spaces, it is shown in \cite{LiangZhou24} that boundary rigidity is a quasi-isometry invariant.
The Morse boundary is also a quasi-isometry invariant and behaves very well under quasi-isometries.
Thus, it is natural to ask whether Morse boundary rigidity is a quasi-isometry invariant.
To this end, we present the following proposition.

\begin{Proposition}
    Let $(X, d_X)$ and $(Y, d_Y)$ be two proper geodesic metric spaces with non-empty Morse boundaries. If $f: X\to Y$ is a quasi-isometry and $Y$ is Morse boundary rigid, then $X$ is Morse boundary rigid.
\end{Proposition}

\begin{proof}

Let \( g: X \to X \) be a quasi-isometry such that the induced homeomorphism on the Morse boundary is the identity, i.e., \( \partial g = \text{id}_{\partial_* X} \). We aim to show there exists a constant \( D > 0 \) satisfying \( \sup_{x \in X} d_X(x, g(x)) < D \).

Let \( f_1: Y \to X \) be a quasi-inverse of \( f \). The composition \( f \circ g \circ f_1: Y \to Y \) is also a quasi-isometry, which induces a homeomorphism \( \partial(f \circ g \circ f_1): \partial_* Y \to \partial_* Y \) on the Morse boundary of \( Y \). By the naturality of the Morse boundary functor, we have:
\[
\partial(f \circ g \circ f_1) = (\partial f) \circ (\partial g) \circ (\partial f_1) = (\partial f) \circ \text{id}_{\partial_* X} \circ (\partial f_1) = \text{id}_{\partial_* Y}.
\]

Since \( Y \) is Morse boundary rigid, there exists a constant \( D_1 > 0 \) such that for all \( y \in Y \):
\[
d_Y\left(y, f\left(g\left(f_1(y)\right)\right)\right) < D_1.
\]
Take an arbitrary \( x \in X \), and set \( y = f(x) \), \( h = f \circ g \), \( x' = f_1(y) \). Suppose \( f \) is a $(K, C)$ quasi-isometry, then
\[
d_X(x, g(x)) \le K d_Y\left(f(x), h(x)\right) + KC.
\]

By the triangle inequality in the metric space \( Y \):
\[
d_Y\left(f(x), h(x)\right) \le d_Y\left(f(x), h(x')\right) + d_Y\left(h(x), h(x')\right).
\]
  and note that
\[
d_Y\left(f(x), h(x')\right) = d_Y\left(y, f\left(g\left(f_1(y)\right)\right)\right) < D_1.
\]
Since \( h = f \circ g \) is a quasi-isometry from \( X \) to \( Y \), there exist constants \( K_1\ge1, C_1 > 0 \) (depending only on \( f \) and \( g \)) such that:
\[
\begin{aligned}
d_X(x, g(x)) 
&\le K \bigl(D_1 + d_Y\bigl(h(x), h(x')\bigr)\bigr) + KC \\
&\le K \bigl(D_1 + K_1 \cdot d_X(x, x') + C_1\bigr) + KC.
\end{aligned}
\]
Since \( f_1 \) is a quasi-inverse of \( f \), \( d_X(x, x') = d_X\left(x, f_1(f(x))\right) \) is uniformly bounded. Let \( M > 0 \) denote this uniform bound.
Let \( D = K \left(D_1 + K_1 M + C_1\right) + KC \). Since \( D \) is a constant independent of \( x \), we conclude \( \sup_{x \in X} d_X(x, g(x)) < D \), as desired.

\end{proof}

\begin{Remark}
   The above proof extends directly to the sublinearly Morse boundary rigid setting, with nearly identical reasoning.
\end{Remark}

 For a finitely generated group $G$ with a finite generating set $S$. The above properties are well-defined for its Cayley graph $\mathcal{G}(G, S)$ with respect to the word metric. Thus, the Morse geodesically rich property, center-exhaustive property, Morse boundary rigidity, and Morse uniform space are well-defined for the group $G$.
 
\subsection*{Applying to Cayley graphs}

\begin{Proposition}\label{cocomp}
    Let $X$ be a proper, cocompact geodesic metric space whose Morse boundary $\partial_*X$ contains at least $3$ points. Then, $X$ is center-exhaustive.
\end{Proposition}
\begin{proof}
    Let $p, q, r\in \partial_*X$ be the three distinct points. Fix an ideal triangle $\triangle(p, q, r)$ and suppose that $(p, q, r)\in \partial_*X^{(3, N)}$ for some Morse gauge $N$. Fix a coarse center $x_0$ of $(p, q, r)$ such that $x_0$ lies within $\delta_N$ of all three sides of some ideal triangle $\triangle(p, q, r)$.
    By the hypothesis that $X$ is cocompact, there exists a group $G$ acting cocompactly by isometries on $X$. 
    There exists a constant $R>0$ such that for any $x\in X$, we have some $g\in G$ such that $d(g(x_0), x)\le R$.
    Isometries preserve the Morse gauges of bi-infinite geodesics. This implies that the point $x$ lies within $R+\delta_N$ of all three sides of some ideal triangle
    $\triangle(p', q', r')$, where $(p', q', r')\in \partial_*X^{(3, N)}$. This completes the proof.

\end{proof}

Theorem~\ref{thm:cocompact} easily comes from Proposition~\ref{cocomp} and Theorem~\ref{thm:equivalence}.

The following proposition is from \cite{Liu21}. We call a finitely generated group $G$ a \textit{non-elementary Morse group} if it is not virtually cyclic and has non-empty Morse boundary.
\begin{Proposition}
    Let $G$ be a non-elementary Morse group. Then its Morse boundary contains infinitely many points.
\end{Proposition}

Thus, we obtain the following corollary, which coincides with Theorem~\ref{groupapp}.

\begin{Corollary}
      Let $G$ be a non-elementary Morse group. Then $G$ is Morse geodesically rich, center-exhaustive, and its Morse boundary $\partial_*G$ is uniformly perfect. In particular, $G$ is Morse boundary rigid.
\end{Corollary}

 As noted in the introduction, Morse groups form a large class encompassing mapping class groups, Out$(F_n)$, CAT(0) groups, hyperbolic and relative hyperbolic groups, Artin groups, Coxeter groups, 3-manifold groups, and related families.

\section{Establishment of Theorem~\ref{homeo}}

    In this section, we establish Theorem~\ref{homeo}. Recall that each of the spaces of $X$ and $Y$ satisfies one of the three equivalent conditions stated in Theorem~\ref{thm:equivalence}. For the sake of convenience in our subsequent arguments, we may assume without loss of generality that both $X$ and $Y$ are center-exhaustive.

\subsection{The extension of \ensuremath{f}}  
    Let $f:\partial_*X\to \partial_*Y$ be a homeomorphism. We next construct a map $\Phi_{f}$ to be an extension of $f$ to $X$. 
    Our construction follows the approach proposed in Section~4 of \cite{Liu22}. To proceed, we first introduce the following notations and auxiliary concepts.

For the purpose of invoking the subsequent proposition without proof, we adopt the terminology from. Let $\Theta_3(\partial_* X)$ be the set of all distinct triples of points in $\partial_*X$. 
For any $(p, q, r)\in \Theta_3(\partial_*X)$, there exists some Morse gauge $N'$ such that $(p, q, r)\in  \partial_*X^{(3, N')}$. 
We define 
\[
k=1+\inf\! \left\{ K \middle|\!\; E_K(p, q, r)\neq \emptyset \text{ for some Morse gauge } N' \right\},
\] and set $E(p, q, r)=E_k(p, q, r)$. It is worth noting that the constant $k$ and the set $E(p,q,r)$ depend only on the triple $(p, q, r)$, and are independent of the choice of $N'$.

For any $(p, q, r)\in \Theta_3(\partial_*X)$, we define a projection map:
\[
\pi_X:\Theta_3(\partial_*X)\to X, 
\]
\[
\pi_X(p, q, r)=x',
\]
where $x'$ is an arbitrary center point chosen from $E(p ,q, r)$. The projection map $\pi_Y$ from $\Theta_3(\partial_* Y)$ to $Y$ can be defined analogously.

 Suppose that $X$ is center-exhaustive, there exist a Morse gauge $N$ and a constant $K\ge0$ such that for any $x\in X$, there exists some $(p, q, r)\in \partial_*X^{(3, N)}$ satisfying $x\in E_K(p, q, r)$.
Set $x'=\pi_{X}(p, q, r)\in E(p, q, r)=E_{k}(p, q, r)$ with $k\le 1+\delta_N$, where $\delta_N$ is a constant associated with the Morse gauge $N$. We thus have
$x, x'\in E_{K_1}(p, q, r)$, where $K_1=\max\{K, 1+\delta_N\}$. 
By Lemma~\ref{Lemma:Center}, there exists a constant $L\ge0$ depending only on $N$ and $K$ such that $d_X(x, x')\le L$. This implies that for any $x\in X$, 
\[
\pi_{X}^{-1}(B(x, L))\cap \partial_*X^{(3, N)}\neq \emptyset,
\]
where $B(x, L)$ denotes the closed ball of radius $L$ centered at $x$ in $X$.

Since $f:\partial_*X\to \partial_*Y$ is a homeomorphism between two Morse boundaries, we have that the set $\pi_Y(f(\pi_{X}^{-1}(B(x, L))\cap \partial_*X^{(3, N)}))$ is nonempty in $Y$.

We now define the map $\Phi_f$ to be an extension of $f$ to $X$ as follows:
\[
\Phi_f:X\to Y,
\]
where for each $x\in X$, $\Phi_f(x)=y$ and $y$ is an arbitrary point chosen from the set
\[
\pi_Y(f(\pi_{X}^{-1}(B(x, L))\cap \partial_*X^{(3, N)})).
\]
Note that the definition of $\Phi_f$ depends on:
\begin{itemize}
    \item the choice of points in the preimage map $\pi_X^{-1}$ and the projection map $\pi_Y$,
    \item the Morse gauge $N$, the constant $L$, and the homeomorphism $f$.
\end{itemize}
The equality $\Phi_f(x)=y$ is equivalent to the existence of a triple $(p, q, r)\in \partial_*X^{(3, N)}$ such that
\begin{equation}\label{extension}
    d_{X}(\pi_X(p, q, r), x)\le L \quad \text{and} \quad \pi_Y(f(p), f(q), f(r))=y.
\end{equation}

\subsection{Propositions for boundedness}

We first state a proposition (originally from Proposition~4.1 of \cite{Liu22}) that plays a critical role in our argument; its proof relies solely on the properties of the homeomorphism $f$. We adapt it to our setting as follows:

\begin{Proposition}\label{key}
    Let $X, Y$ be two proper geodesic metric spaces, and suppose $X$ is center-exhaustive. Let $f:\partial_*X\to \partial_*Y$ be a homeomorphism that is quasi-symmetric, bi-H\"older, or strongly quasi-conformal. Then for any Morse gauge $N_1$, there exists a function $\eta_{N_1}:(0, \infty)\to(0, \infty)$ such that for all $(p, q, r), (p',q',r')\in \partial_*X^{(3, N_1)}$, 
    \[
    \begin{aligned}
   d_X\!\bigl(\pi_X(p, q, r), \pi_X(p', q',r')\bigr) \le \theta &\implies \\
d_Y\!\bigl(\pi_Y(f(p), f(q), f(r)), \pi_Y(f(p'),f(q'),f(r'))\bigr) &\le \eta_{N_1}(\theta).
    \end{aligned}
    \]

\end{Proposition}

Fix the Morse gauge $N$ and constant $L$ used in the construction of $\Phi_f$. Combining these with Proposition~\ref{key}, we obtain the following boundedness result for the extension map:

\begin{Proposition}\label{extenbound}
   For any $t\ge0$, there exists a constant $D=D(t, N, L, f)$ such that for all $x_1, x_2\in X$ with $d_X(x_1, x_2)\le t$, we have
    \[
    d_Y(\Phi_f(x_1), \Phi_f(x_2))\le D.
    \]
\end{Proposition}

\begin{proof}
    For $i=1,2$, choose a triple $(p_i, q_i, r_i)\in\pi_X^{-1}(B(x_i, L))\cap \partial_*X^{(3, N)}$ (nonempty by the construction of $\Phi_f$) such that 
    \[
    \pi_Y(f(p_i), f(q_i),f(r_i))=\Phi_f(x_i).
    \]
    By Equation~\eqref{extension}, we have $d_X(x_i, \pi_X(p_i,q_i,r_i))\le L$ for $i=1,2$. Applying the triangle inequality in $X$, we get
    \[
d_X\!\bigl(\pi_X(p_1, q_1, r_1), \pi_X(p_2,q_2, r_2)\bigr) \le d_X(x_1,x_2) + \sum_{i=1}^2 d_X\!\bigl(x_i,\pi_X(p_i,q_i,r_i)\bigr) \le t+2L.
    \]
    Since $(p_1,q_1,r_1), (p_2,q_2,r_2)\in\partial_*X^{(3, N)}$, Proposition~\ref{key} (with $N_1=N$ and $\theta=t+2L$) implies
    \[
d_Y\!\bigl(\pi_Y(f(p_1),f(q_1),f(r_1)), \pi_Y(f(p_2),f(q_2),f(r_2))\bigr) \le \eta_N(t+2L).
    \]
    By the definition of $\Phi_f$, this gives
    \[
    d_Y(\Phi_f(x_1), \Phi_f(x_2))\le D,
    \]
    where $D=\eta_N(t+2L)$ (and thus $D$ depends only on $t, N, L,$ and $f$, as required).
\end{proof}

    When the homeomorphisms $f$ and $f^{-1}$ are $2$-stable, and quasi-M\"obius, the validity of the preceding propositions was established in \cite{CharneyCordesMurray19} and the claim in the subsequent section holds for free. Although the map $\pi_X$ differs slightly from the construction in that work, the resulting discrepancies are negligible. Accordingly, it suffices to consider the remaining cases.

\subsection{Proof of Theorem~\ref{homeo}}

\begin{proof}
Let $L$ and $N$ denote the fixed constant and Morse gauge, respectively, used in the construction of $\Phi_f$.
Since $Y$ is also center-exhaustive, we replicate the identical construction for $f^{-1}$: we denote by $\Phi_{f^{-1}}$ the map on $Y$ induced by the extension of $f^{-1}$. 
Specifically, there exist a constant $L'\ge0$ and a Morse gauge $N'$ such that 
\[
\Phi_{f^{-1}}:Y\to X,
\]
where for any $y\in Y$, we set $\Phi_{f^{-1}}(y)=x$ with $x$ being an arbitrary point selected from the nonempty set
\[
\pi_X(f^{-1}(\pi_{Y}^{-1}(B(y, L'))\cap \partial_*Y^{(3, N')})).
\]

By Proposition~\ref{extenbound} and following the method of Theorem~4.3 in \cite{Liu22}, we immediately obtain the following linear boundedness: 
\begin{itemize}
    \item There exists a constant $A=\eta_{N}(1+2L)>0$ (depending only on $N, L,$ and $f$) such that for all $x_1, x_2\in X$,
    \[
    d_Y(\Phi_f(x_1), \Phi_f(x_2))\le A d_X(x_1, x_2)+A.
    \]
    \item Similarly, there exists a constant $A'>0$ (depending only on $N', L',$ and $f^{-1}$) such that for all $y_1, y_2\in Y$,
    \[
    d_X(\Phi_{f^{-1}}(y_1), \Phi_{f^{-1}}(y_2))\le A' d_Y(y_1, y_2)+A'.
    \]
\end{itemize}

We now establish that $\Phi_f$ and $\Phi_{f^{-1}}$ are \textit{quasi-inverses}. 
Let $y\in Y$, set $x=\Phi_{f^{-1}}(y)$ and $y_1=\Phi_{f}(x)$. Choose triples 
\[
(p_1, q_1, r_1)\in \pi_X^{-1}(B(x, L))\cap \partial_*X^{(3, N)} \quad \text{and} \quad (f(p), f(q), f(r))\in \pi_Y^{-1}(B(y, L'))\cap \partial_*Y^{(3, N')}
\]
such that
\[
\pi_Y(f(p_1),f(q_1), f(r_1) )=y_1\quad\text{and}\quad\pi_X(p, q, r)=x.
\]
Define $y_0=\pi_Y(f(p), f(q), f(r))$ and $x_1=\pi_X(p_1, q_1, r_1)$.
By construction, we have 
\[
d_Y(y, y_0)\le L' \quad \text{and} \quad d_X(x,x_1)\le L.
\]

 \textbf{Claim:} $(p, q, r)\in \partial_*X^{(3,N_1)}$ for some Morse gauge $N_1=N_1(f^{-1}, N')$.
 
 We defer the proof of this claim to the end of the argument. Assuming the claim holds, let $N_2=\max\{N_1, N\}$; then $(p, q, r), (p_1, q_1,r_1)\in \partial_*X^{(3, N_2)}$. By Proposition~\ref{key}, we have 
 \[
 d_Y(y_0, y_1)\le\eta_{N_2}(L).
 \]
 Applying the triangle inequality in $Y$, it follows that,
 \[
 d_Y(y, y_1)\le d_Y(y, y_0)+d_Y(y_0, y_1)\le L'+\eta_{N_2}(L).
 \]
Thus, there exists a constant $L_1=L'+\eta_{N_2}(L)$ (depending only on  $N, N', f, f^{-1}, L,$ and $L'$) such that for any $y\in Y$,  
\[
d_Y(y, \Phi_f(\Phi_{f^{-1}}(y)))\le L_1.
\]
By a symmetric argument, $d_X(x, \Phi_{f^{-1}}(\Phi_{f}(x)))$ is uniformly bounded for all $x\in X$. Therefore, $\Phi_f$ and $\Phi_{f^{-1}}$ are quasi-isometries. 

The fact that $\Phi_f$ induces $f$ on the Morse boundary follows directly from the proof of Theorem~4.3 in \cite{Liu22}.

\medskip
\paragraph{Proof of the Claim} 
By definitions of quasi-conformal, quasi-symmetric, and bi-H\"older homeomorphisms, the map $f^{-1}$ is a \textit{basetriangle stable map} (see Definition~3.4 in \cite{Liu22}) with respect to the Morse gauge $N'$. 

Since $(f(p), f(q), f(r))\in \partial_*Y^{(3, N')}$ and $y_1\in E(f(p), f(q), f(r))$, Lemma~3.2 in \cite{Liu22} gives $f(p), f(q), f(r)\in \partial Y_{y_1}^{(M)}$ for some Morse gauge $M=M(N')$.

For this Morse gauge $M$, the basetriangle stability of $f^{-1}$ implies there exists a Morse gauge $M_1=M_1(N', M)$ such that for the triple $(f(p), f(q), f(r))$ (which lies in $\partial_*Y^{(3, N')}$) and points $x=\Phi_{f^{-1}}(y)\in E(p, q, r)$, $y_1\in E(f(p), f(q), f(r))$, we have 
\[
f^{-1}(\partial Y_{y_1}^{(M)})\subset \partial X_{x}^{(M_1)}.
\]
Thus, $p, q, r\in \partial X_x^{(M_1)}$. By Lemma~\ref{relation}, there exists a Morse gauge $M_2=M_2(M_1)$ such that $p, q, r\in \partial_*^{M_2}X_x$. 
The geodesics $[x, p],[x, q],[x, r]$ are $M_2$-Morse; by Lemma~\ref{slim-M}, all bi-infinite geodesics $[p, q], [p,r],$ and $[q, r]$ are $N_1$-Morse, where $N_1$ depends only on $M_2$. 

Since $M_2$ depends only on $M_1$, which in turn depends only on $N'$, $M$, and $f^{-1}$ (where $M$ depends only on $N'$), we conclude that $N_1$ depends only on $f^{-1}$ and $N'$. This conclusion relies crucially on both the basetriangle stability of $f^{-1}$ and the center-exhaustiveness of $Y$, completing the proof of the claim.

\end{proof}

\bibliography{Ref}

\bigskip
{School of Mathematics, Hunan University, Changsha, Hunan, 410082, P.R.China}

{\tt Email: hansz@hnu.edu.cn}

\bigskip
{School of Mathematical Sciences \& LPMC, Nankai University, Tianjin 300071, P.R.China}

{\tt Email: qingliu@nankai.edu.cn}

\end{document}